\documentclass[11pt,a4paper]{amsart}
\usepackage{amsmath,amsfonts,amsthm,amsopn,color,amssymb,enumitem}
\usepackage{palatino}
\usepackage{graphicx}
\usepackage{caption}
\usepackage{subcaption}
\usepackage[colorlinks=true]{hyperref}
\hypersetup{urlcolor=blue, citecolor=red, linkcolor=blue}

\usepackage{cite}
\usepackage{esint}
\usepackage[title]{appendix}

\newtheorem{theorem}{Theorem}[section]
\newtheorem{lemma}[theorem]{Lemma}

\newtheorem{definition}[theorem]{Definition}
\newtheorem{proposition}[theorem]{Proposition}
\newtheorem{remark}[theorem]{Remark}

\newcommand{\Sd}{\mathbb{S}^2}

\newcommand{\e}{\varepsilon}

\newcommand{\R}{\mathbb{R}}

\newcommand{\D}{\mathbb{D}}
\newcommand{\oD}{\overline{\mathbb{D}}}

\newcommand{\weakto}{\rightharpoonup}

\renewcommand{\b }{\beta }

\newcommand{\G}{\Gamma}

\renewcommand{\S}{\Sigma}

\def\bbm[#1]{\mbox{\boldmath $#1$}}
\newcommand{\beq }{\begin{equation}}
\newcommand{\eeq }{\end{equation}}

\newcommand{\tK}{\tilde{K}}
\def\sideremark#1{\ifvmode\leavevmode\fi\vadjust{\vbox to0pt{\vss% the remark3
 \hbox to 0pt{\hskip\hsize\hskip1em%                          will appear only
 \vbox{\hsize3cm\tiny\raggedright\pretolerance10000%          on the side
  \noindent #1\hfill}\hss}\vbox to8pt{\vfil}\vss}}}%

\begin{document}

\title[Conformal metrics of the disk with prescribed curvatures]{Conformal metrics of the disk with prescribed Gaussian and geodesic curvatures}

\author{David Ruiz}
  \address{David Ruiz \\
    IMAG, Universidad de Granada\\
    Departamento de An\'alisis Matem\'atico\\
    Campus Fuentenueva\\
    18071 Granada, Spain}
  \email{daruiz@ugr.es}

\thanks{D. R. has been supported by the FEDER-MINECO Grant PGC2018-096422-B-I00 and by J. Andalucia (FQM-116). He also acknowledges financial support from the Spanish Ministry of
		Science and Innovation (MICINN), through the \emph{IMAG-Maria de Maeztu} grant
		CEX2020-001105-M/AEI/10.13039/501100011033.}

% Type down your paper title



% The abstract
\begin{abstract}
{This paper is concerned with the existence of conformal metrics of the disk with prescribed Gaussian and geodesic curvatures. Being more specific, given nonnegative smooth functions $K: \oD \to \R$ and $h: \partial \D \to \R$, we consider the problem of finding a conformal metric realizing $K$ and $h$ as Gaussian and geodesic curvatures, respectively. This is the natural analogue of the classical Nirenberg problem posed on the disk. As we shall see, both curvatures play a role in the existence of solutions. Indeed we are able to give existence results under conditions that involve $K$ and $H$, where $H$ denotes the harmonic extension of $h$. The proof is based on the computation of the Leray-Schauder degree in a compact setting.}
\end{abstract}

\maketitle

\section{Introduction}
\setcounter{equation}{0}

Let $(\S,g)$ denote a compact surface $\S$ equipped with a certain metric $g$. The classical Kazdan-Warner problem (see \cite{ber, KW})  consists of determining whether a prescribed function $K$ is the Gaussian curvature of a new metric $\tilde{g}$ conformal to $g$. We point out that if $K$ is a constant function of the same sign of $\chi(\S)$, then this problem is solvable via the Uniformization Theorem.

\medskip If $\tilde{g}=e^u g$ and $K$, $K_g$ are respectively the Gaussian curvatures relative to these metrics, then the following relation holds:
\[
-\Delta_g u + 2K_g(x) = 2K(x)e^u \quad\mbox{in}\quad \S.
\]
Hence the Kazdan-Warner problem reduces to solving this equation. This problem has been completely solved in the case of the torus (\cite{KW}) and the projective plane (\cite{Moser}).

%To start with, the classical \textit{Uniformization Theorem} asserts that every simply connected Riemann surface is conformally equivalent to one of the following three Riemann surfaces: the open unit disk, the complex plane or the Riemann sphere. As a consequence, we can conclude that every compact orientable surface carries a conformal metric with constant Gaussian curvature. Hence we can assume from now on that $K_g$ is constant. 
%

%Integrating \eqref{eq:regu} and taking into account the Gauss-Bonnet Theorem, we obtain that
%\beq \label{bonnet0}
%\int_{\Sigma} K_g \, dV_g= \int_{\Sigma} K e^v \, dV_g = 2 \pi \chi(\Sigma),
%\eeq
%where $\chi(\Sigma)$ is the Euler characteristic of $\S$. By this formula, we can observe how the topology of the surface $\Sigma$ gives necessary conditions on the choice of the function $K$. Indeed, the sign of the function $K$ in at least some point of $\S$ is prescribed by $\chi(\Sigma)$.

\medskip 

The case of the standard sphere $\mathbb{S}^2$ is particularly delicate and receives the name of Nirenberg problem. As commented above, this is equivalent to solve:

\begin{equation}\label{eq:Nirenberg}
-\Delta u + 2 = 2K(x)e^{u} \quad\mbox{in}\quad \mathbb{S}^2.
\end{equation}

This question has been addressed by a large number of papers (see for instance \cite{Aubin, ChGYg, ChYgActa, ChYg, chen-ding, chen-ding2, CL-CPAM, espinar,  Ji, hamza, Han, han-li, KW, Moser, struwe, xu-yang}), some of whom are commented below. For instance, an obstruction to existence was found by Kazdan and Warner in \cite{KW}. On the other hand, if $K$ has antipodal symmetry (i.e. $K(-x) = K(x)$) and it is somewhere positive then there exists a solution, see \cite{Moser}. Other results under symmetry assumptions are given in \cite{chen-ding}. In \cite{struwe}  solutions of \eqref{eq:Nirenberg} have been found as the asymptotic limit of a suitable parabolic flow.

%Kazdan and Warner deduce that a solution of \eqref{eq:regu} verifies the following necessary condition
%\beq\label{KazWarCond}
%\int_{\Sd}\nabla \zeta \cdot \nabla K e^u dV_{g_0} = 0, \quad \mbox{ where } \quad -\Delta_{g_0} \zeta = 2\zeta \quad \mbox{ in $\Sd$}.
%\eeq
%
%\bigskip 

A related and significant issue in this kind of problems is the study of compactness of solutions, starting from \cite{breme, li-sha}. Being more specific, let us consider a sequence of solutions of the problem:
\[
-\Delta u_n + 2 = 2K_n(x)e^{u_n} \quad\mbox{in}\quad \mathbb{S}^2,
\]
where $K_n$ converges in $C^2$ sense to a strictly positive function $K(x)$. Observe that if $K_n=1$, the problem is invariant by the group of conformal maps of the sphere, which is not compact. In the non-constant case, this invariance is lost but concentration of solutions may still occur. This is the so-called ``bubbling phenomena"; the solutions concentrate all their mass around a certain singular point. 

Not only the asymptotic behavior of the sequence is relevant, but also the location of the point of concentration. It has been shown (\cite{ChYg2, ChGYg}) that if the sequence $u_n$ is blowing-up, then the point of concentration satisfies:
\begin{equation} \label{cond-nir}  \nabla K(p)=0,\ \Delta K(p)=0.  \end{equation}

Regarding existence of solutions, the first general result (without symmetry requirements) was given for the first time in \cite{ChYgActa}, and had a great impact in the research community.

\begin{theorem}[\cite{ChYgActa}] \label{teoA} Let $K: \Sd \to \R^+$ be a $C^2$ function such that $|\nabla K(p)| + |\Delta K(p)| \neq 0$ for all $p \in \Sd$. Assume also the critical points of $K$ are nondegenerate. Then there exists a solution of problem \eqref{eq:Nirenberg} provided that:
	 
\begin{equation} \label{ChYg} 1- \sum_{x \in S_{-}} (-1)^{ind(x)}  \neq 0. \end{equation}
	
	Here $S_{-}=\{x \in \Sd: \ \nabla K(x)=0, \ \Delta K(x)<0\}$ and $ind(x)$ denotes its Morse index.
	
\end{theorem}

The proof uses a quite complicated min-max technique and Morse theory, see also \cite{ChL, Han}. As a drawback, it requires the function $K$ to be a Morse function. 

A somewhat different approach was given in \cite{Ji}, more in the spirit of \cite{ChGYg, ChYg2}, where the nondegeneracy condition on the critical points of $K$ is dropped. In order to state this result, let us define:

\begin{equation} \label{defG}G: \Sd \to \R^3,\ G(x)= \nabla K(x) + \Delta K(x) x. \end{equation}

Observe that condition \eqref{cond-nir} is not satisfied for any $p \in \Sd$ if and only if $0 \notin G(\S^2)$, and hence in such case we have compactness of solutions.

\begin{theorem}[\cite{Ji}] \label{teoB}Let $K: \Sd \to \R^+$ be a $C^2$ function such that $|\nabla K(p)| + |\Delta K(p)| \neq 0$ for all $p \in \Sd$. Then there exists a solution to problem \eqref{eq:Nirenberg} provided that:
	
\[  d_{B} (G, \Sd, 0) \neq 0. \]
	
where $d_B$ denotes the usual Brouwer degree. 
	
	\end{theorem}
The proof is based on the computation of the Leray-Schauder degree of the problem. We point out that if $K$ is a Morse function, the Brouwer degree of $G$ coincides precisely with the left hand side of formula \eqref{ChYg}, see \cite[Corollary 1.2]{Ji}.

\medskip

If $\S$ is a surface with boundary, one usually imposes boundary conditions. A natural geometric problem consists in prescribing also the geodesic curvature of the boundary; in this way we are led to the problem:
\begin{equation} \label{gg0}
\left\{\begin{array}{ll}
-\Delta u +2 K_g(x) = 2 K(x) e^u  & \text{ in } \S, \\
\frac{\partial u}{\partial \nu} +2 h_g(x) = 2h(x)e^{u/2}  &\text{ on } \partial\S,\end{array}\right.
\end{equation}
where $\nu$ is the outward normal vector to $\partial \S$. Here $h_g$ is the  geodesic curvature of $\partial \S$ under the metric $g$, and $h$ is the geodesic curvature to be prescribed for the new metric $\tilde{g}=e^u g$. 

\medskip

In the literature there are some concerning problem \eqref{gg0}.  The case of constant $K$, $h$ has been considered in \cite{brendle}, where the author used a parabolic flow to obtain solutions in the limit. Some classification results for the half-plane are also available in \cite{galvez, li-zhu, Zhang}. The case of nonconstant curvatures was addressed for the first time in \cite{cherrier}, but the results there are partial and in some of them an unknown Lagrange multiplier appears in the equation. In \cite{lsmr} the case of surfaces topologically different from the disk is studied when $K<0$. In that paper a new type of blow-up phenomenon appears, where length and area diverge. 

\medskip

The case of the disk $\S = \D$ can be seen as a natural generalization of the Nirenberg problem. Indeed, also here the effect of the noncompact group of conformal maps of the disk plays a fundamental role. The problem becomes:
\begin{equation} \label{gg}
\left\{\begin{array}{ll}
\displaystyle{-\Delta u = 2K(x)e^u } \qquad & \text{ in }\D, \\
\displaystyle{\frac{\partial u}{\partial \nu} +2 = 2h(x)e^{u/2}} \qquad  &\text{ on } \partial \D.
\end{array}\right.
\end{equation}

Integrating \eqref{gg} and applying the Gauss-Bonnet Theorem, one obtains
\[ 
\int_{\D}K(x) e^u \, +  \int_{\partial \D}h(x) e^{u/2} \, =  2\pi,
\]
which implies that either $K$ or $h$ must be positive somewhere.  

Some works have also addressed problem \eqref{gg}. For example, the case $h=0$ and $\partial_{\nu}K=0$ has been treated in \cite{ChYg} passing to a Nirenberg problem via reflection (see also \cite{guo-hu}). The case $K=0$ has attracted more attention, see \cite{DLMR, GuoLiu, kcchang, LiLiu, LiHu}. More recently, in \cite{manuela} a flow approach is used to give existence solutions.

Up to our knowledge, very few existence results are availabe on equation \eqref{gg} when both curvatures are non-constant. One of them is \cite{CruzRuiz}, where an existence result for \eqref{gg} with positive symmetric curvatures is given, in the spirit of the aforementioned result of Moser (\cite{Moser}).

Regarding the blow-up analysis of solutions, the case $K=0$ has been covered in \cite{GuoLiu}, whereas in \cite{DLMR} a new approach is given under mild conditions on the sequence of functions $h_n$. For the general case, a rather complete result has been recently given in \cite{jlmr}. We state it in a version which is adequate for our purposes (see Theorem 1.1 and Remark 1.2 in \cite{jlmr}).

\medskip

%From now on we will assume that $K$ and $h$ are constantly signed, namely that they do change sign.

%In this paper we study the following problem:
%
%\begin{equation} \label{gg}
%\left\{\begin{array}{ll}
%\displaystyle{-\Delta u = 2Ke^u } \qquad & \text{in $\D$}, \\
%\displaystyle{\frac{\partial u}{\partial \nu} +2 = 2he^{u/2}} \qquad  &\text{on $\partial \D $.}
%\end{array}\right.
%\end{equation}
%
%Equation \eqref{gg} arises in the problem of prescribing the Gaussian and geodesic curvature on a surface and its boundary, respectively, via a conformal change of the metric. More precisely, consider a conformal change of the metric ${g}=e^u dz$. Then, the new Gaussian and geodesic curvatures $K$ and ${h}$ with respect to $g$ satisfy \eqref{gg}.
%
%Integrating \eqref{gg} and applying the Gauss-Bonnet theorem, one obtains
%\beq\label{GB-bis}
%\int_{\D}K e^u \, +  \int_{\partial \D}h e^{u/2} \, =  2\pi.
%\eeq

%\begin{remark} The value $\Phi$ solves the equation: $\Phi^2 - 2 H \Phi - K=0$.
%\end{remark}

\begin{theorem}[\cite{jlmr}] \label{blowup} Let $u_n$ be a sequence of solutions of the problem
	\[
	\left\{\begin{array}{ll}
	\displaystyle{-\Delta u_n = 2K_n(x)e^{u_n} } \qquad & \text{in $\D$},\\
	\displaystyle{\frac{\partial u_n}{\partial \nu} +2 = 2h_n(x)e^{u_n/2}} \qquad  &\text{on $\partial \D $},
	\end{array}\right.
	\]
where $0 \leq K_n \to K$ in $C^2(\oD)$ and $0\leq h_n \to h$ in $C^2(\partial \D)$ as $n\to +\infty$. We shall assume also that either $h$ or $K$ is strictly positive. If $u_n$ is a blowing-up sequence, namely, $\sup \{ u_n \} \to +\infty$, then there exists a unique point $p \in \partial \D$ such that 
	\begin{enumerate}
	\item[i)]	
	$
u_n\to -\infty \qquad \mbox{ locally uniform in } \oD\setminus \{p\}, 
$

\item[ii)]
\[
h_ne^{u_n/2} \weakto \b \delta_p  \quad \mbox{ and } \quad K_ne^{u_n} \weakto (2\pi -\b) \delta_p,
\]
in the sense of measures, where
\[
\displaystyle{\beta:=2\pi \frac{h(p)}{\sqrt{h^2(p)+K(p)}} }.
\]

		\item[iii)] The point $p$ satisfies 
		
		 \beq \label{cond-nir2} \nabla \Phi(p)=0, \eeq
where 
		
\begin{equation}\label{PHI}
\Phi(x)= H(x) + \sqrt{H(x)^2 + K(x)},
\end{equation}
and $H$ denotes the harmonic extension of $h$, that is, 

\begin{equation*} 
\left\{\begin{array}{ll}
\displaystyle{\Delta H = 0} \qquad & \text{in $\D$},\\
H(x) = h(x) \qquad  &\text{on $\partial \D $.}
\end{array}\right.
\end{equation*}
	\end{enumerate}
\end{theorem}

Let us emphasize that the information on the location of the blow-up point is encoded in the function $\Phi$, which depends on $K$ locally but it has a nonlocal dependence on $h$. This interesting fact does not have an analogue in the classical Nirenberg problem.

\medskip 

Prompted by this result, a blowing-up sequence of solutions for problem \eqref{gg} have been constructed via a version of the Lyapunov-Schmidt reduction, see \cite{angela}.

The main goal of this paper is to give the following existence result for equation \eqref{gg}, which shows a certain analogy with Theorem \ref{teoB}. 

\begin{theorem} \label{main} Let $K \in C^2(\overline{\D})$, $h \in C^2(\partial \D)$, with 

$$K \geq 0,\ h > 0 \mbox{ or } K >0, \ h\geq 0.$$

Assume that 

\begin{equation} \label{assumption} \nabla \Phi(x) \neq 0 \ \forall \ x \in \partial \D \, \mbox{ and } \ deg_B (\nabla \Phi, \D, 0) \neq 0, \end{equation} where $\Phi$ is defined in \eqref{PHI}. Then there exists a solution for \eqref{gg}.

\end{theorem}

If $K=0$, then $\Phi = 2H$, and the result above is consistent with the results of \cite[Theorem 1']{kcchang}. Indeed, any nondegenerate critical point of $H$ in  $\D$ gives a contribution of $-1$ to the Brouwer degree. Also \cite[Theorem 1]{kcchang} can be seen under our perspective, but this is not straightforward. This is related to the fact that, for the Nirenberg problem, formula \eqref{ChYg} gives the Brouwer degree of $G$ as defined in \eqref{defG}. It is important to point out here that the conjugate function defined in \cite{kcchang} is nothing but the normal derivative of the harmonic extension.

The connection of our work with \cite{manuela} remains obscure to the author.

The proof is based on the computation of the Leray Schauder degree of the problem, properly formulated. Indeed we shall see that the Leray Schauder coincides, possibly up to a sign, with the Brouwer degree of $\nabla \Phi$, which implies the result. This is done in several steps. In a first step, we use Theorem \ref{blowup} to pass via a homotopy to a problem with $h=0$, that is,

\begin{equation} \label{gg00}
\left\{\begin{array}{ll}
\displaystyle{-\Delta u = 2 \tK(x)e^u } \qquad & \text{ in }\D, \\
\displaystyle{\frac{\partial u}{\partial \nu} +2 = 0} \qquad  &\text{ on } \partial \D,
\end{array}\right.
\end{equation}
where $\tK= \Phi^2$. The main advantage in doing so is that \eqref{gg00} admits a nice variational formulation as a mean field problem. Problem \eqref{gg0} admits also a variational formulation but it is much more intricate, see \cite{CruzRuiz}. 

The related variational functional $J$ is bounded from below by the Moser-Trudinger inequality, but it fails to be coercive. The heuristic idea of the rest of the proof is as follows: we first fix $a \in \D$ and minimize $J(u)$ under the constraint:

$$ P(u)= \frac{\int_{\D} e^u x }{\int_{\D} e^u} =a.$$

This is the approach of \cite{Ji}, which is also inspired in the previous works \cite{Aubin, ChGYg, ChYg2}. Under this restriction, the functional $J$ becomes coercive and a minimizer $u_a$ is found. In so doing, we can solve equation \eqref{gg00} with an additional Lagrange multiplier $\mu(a) \in \R^2$. The main observation of this paper is that, as $a \to \partial \D$, then $u_a$ forms a blowing-up sequence of solutions. By Theorem \ref{blowup}, we have that  $\mu(a) \to - \nabla (\tilde{K})(a)$. Hence, if the degree of $\nabla \tilde{K}$ is nonzero, we can find $a\in \D$ such that $\mu(a)=0$. The argument of \cite{Ji} is different at this point and uses a previous work of the author \cite{preJi}.

\medskip

Of course there are several obstacles to perform this program. The more important is that the minimizer $u_a$ needs not be unique, and hence the Lagrange multiplier $\mu(a)$ is not continuously defined. We can overcome this difficulty by using another homotopy to pass to a problem where $\tK$ is close to $1$, and in such case a perturbation argument allows us to prove the uniqueness of the constrained minimizer. 

At the end, we need to compute the Leray-Schauder degree of such problem (and not only the existence of solution), in order to show that the existence of solution is preserved along the homotopies. This is rather technical, and uses the invariance properties of the Leray-Schauder degree under homotopies and homeomorphisms.

The rest of the paper is organized as follows. In Section 2 we present some preliminary results and we establish the functional setting and the variational formulation of the problem. Section 3 is devoted to the study of the constrained minimization problem, the existence of minimizer and, if $\tK$ is close to $1$, its uniqueness. The behavior of the corresponding Lagrange multiplier is also studied. In Section 4 we conclude the computation of the Leray-Schauder degree of the problem.

\section{Preliminaries} \label{sec:prelim}
\setcounter{equation}{0}

In this section we collect some preliminary results and establish the functional setting and the variational formulation of the problem.

\subsection{Kazdan-Warner identities and estimates for the case $h=0$.}

As commented in the introduction, in this paper we will first reduce our problem to the case $h=0$. In this subsection we recall some information that appears in \cite{jlmr} adapted to such situation. The first result is the following Kazdan-Warner identity, which is just contained in \cite[Lemma 2.6 and Proposition 2.7]{jlmr}).

\begin{proposition} \label{kw} Let $u$ be a solution of the \eqref{gg00}. Then:
	
	\begin{equation} \label{kw1} \int_{\D} e^u \, \nabla \tK \cdot \tau =0, \mbox{ and}\end{equation}
	\begin{equation} \label{kw2} \int_{\D} e^u \, \nabla \tK \cdot F = 0,\end{equation}
	where $\tau(x_1,x_2)= (-x_2,x_1)$ and $F(x_1,x_2)=(1-x_1^2+x_2^2, -2x_1 x_2)$ (in complex notation, $\tau(x)= ix $ and $F(x)= 1-x^2$). 
	
\end{proposition}

\medskip 

The following asymptotic estimate will also be of use:

\begin{lemma} \label{estima} Let $u_n$ be a sequence of solutions of the problem
	\begin{equation} \label{ggn2}
	\left\{\begin{array}{ll}
	\displaystyle{-\Delta u_n = 2\tK_n(x)e^{u_n} } \qquad & \text{in $\D$},\\
	\displaystyle{\frac{\partial u_n}{\partial \nu} +2 = 0} \qquad  &\text{on $\partial \D $},
	\end{array}\right.
	\end{equation}
	where $\tK_n \to \tK>0$ in $C^2(\oD)$. If $u_n$ is a blowing-up sequence, namely, $\sup \{ u_n \} \to +\infty$, then there exists a unique point $p \in \partial \D$  such that the following statements hold true.
	
\begin{enumerate}
	
	\item[a)] $\tK_ne^{u_n} \weakto 2\pi \delta_p$.
	\item[b)] Assume for simplicity that $p=(1,0)$. There exists a sequence  $\lambda_n \in (0,1)$, $\lambda_n \to 1$, such that, for any $G \in C^2(\oD)$, 
	$$ \int_{\D} e^{ u_n} \nabla G \cdot F = 4 \pi \frac{\partial_{\nu}G(p)}{\tK(p)} (1-\lambda_n)+ o(1-\lambda_n),$$
	where $F$ is defined in Proposition \ref{kw}.
\end{enumerate}
\end{lemma}

Observe that under the assumptions of Lemma \ref{estima}, and thanks to \eqref{kw1}, we conclude that $\nabla \tK(p)=0$, which is a particular case of \eqref{cond-nir2} for the case $h=0$.

\begin{proof}

The proof is basically contained in \cite{jlmr}. Indeed, $\lambda_n$ appears in the profile description of \cite[(5.4)]{jlmr}, see also Lemma 5.2 there. We recall the estimate of the term $(5.7)$ in \cite{jlmr}, and replace the partial derivatives of $K$ with those of $G$.  This leads to the term $ (II)$ in the proof of \cite[Theorem 5.3]{jlmr}; see in particular the asymptotic expression of $II_{21}$ there. By making $h=0$ we can conclude the expression given in b).

\end{proof}

\subsection{Functional setting, and simplification of the problem}

We define $X= \{u \in H^1(\D):\ \int_{\D} u=0 \}$, with the scalar product:

$$ \langle u, v \rangle = \int_{\D} \nabla u \cdot \nabla v.$$

Thanks to the Poincar\'{e} inequality, this scalar product gives us a norm $\| \cdot \|$ equivalet to the usual norm in the Sobolev space $H^1(\D)$.

We identify $X$ with its dual $X^*$ via the Riesz representation, that is, any element $u \in X$ is identified with the linear mapping $u: X \to \R$, $u(v) = \langle u, v \rangle$. 

We define the continuous and nonlinear map:

\begin{equation} \label{def-c} c : X \to \R \mbox{ such that } \int_{\D} K e^{u+c(u)}+ \int_{\partial \D} h e^{(u+c(u))/2} =2 \pi.\end{equation}

Indeed, $e^{c(u)}$ by the assumptions on $K$, $h$, the map

$$ c \mapsto e^c \int_{\D} K e^{u}+ e^{c/2}\int_{\partial \D} h e^{u/2}$$

is strictly increasing and takes all positive values exactly once.

We are now in conditions to define

\begin{equation} \begin{array}{c} \label{def-T} T: X \to X,\ T(u)=v,\ \mbox{ where } v \in X \mbox{ solves:} \\   \\
\left\{\begin{array}{ll}
\displaystyle{-\Delta v = 2Ke^{u+c(u)} } \qquad & \text{ in }\D, \\
\displaystyle{\frac{\partial v}{\partial \nu} +2 = 2he^{(u+c(u))/2}} \qquad  &\text{ on } \partial \D.
\end{array}\right. \end{array}
\end{equation}
 
Observe that by \eqref{def-c} ther exists such a function $v \in X$ and is unique. 

It is not difficult to check that $T$ is a compact operator. Indeed, if $u_n \rightharpoonup u$ in $X$, then $e^{u_n} \to e^u$ in $L^1(\D)$ and $e^{u_n/2} \to e^{u/2}$ in $L^1(\partial \D)$. As a consequence $c(u_n) \to c(u)$. By standard regularity one concludes that $T(u_n) \to T(u) $ in $X$.

Let us now observe that the fixed points of the operator $T$ give rise to solutions of \eqref{gg}. Indeed, if $T(u)=u$, then $u+c(u)$ is a solution to \eqref{gg}.

In this paper we will be concerned with the computation of the Leray-Schauder degree of the map $I-T$ in a ball $B_R$ of sufficiently large radius $R$.

Of course the maps $c$, $T$ defined above depend on the functions $K$, $h$. We shall write $c_{K,h}$ and $T_{K,h}$ to highlight this dependence.

Our first result allows us to consider, in what follows, the case where $h=0$.

\begin{proposition} \label{h=0} For a sufficiently large radius $R$, we have that:
	$$ deg_{LS} (I-T_{K,h}, B_R, 0) = deg_{LS}(I-T_{\tilde{K},0}, B_R,0),$$ where $\tilde{K}= \Phi^2$. Here $deg_{LS}$ stands for the usual Leray-Schauder degree and $\Phi$ is defined as in \eqref{PHI}.
\end{proposition}

\begin{proof} The proof uses a convenient homotopy. For any $s \in [0,1]$, define 
$$ h_s(x)= sh(x).$$

Of course the corresponding harmonic extension satisfies $H_s(x)= s H(x)$. 
Since the blow-up analysis of Theorem \ref{blowup} depends on the map $\Phi$ defined in \eqref{PHI}, we plan to define $K_s$ so that $\Phi_s$ remains constant, that is:

$$ \Phi_s = H_s + \sqrt{H_s^2 + K_s}= H + \sqrt{H^2 + K}= \Phi.$$

The algebra is nice here and an easy computation gives the right definition of $K_s$:

$$ K_s(x)= s K(x) + (1-s) \Phi^2(x).$$

Let us point out that $h_s \geq 0$, $K_s \geq 0$ and at least one of them is strictly positive, for any $s \in [0,1].$

Observe moreover that the operator $T_{K_s, h_s}$ varies continuously on the variable $s$. By Theorem \ref{blowup} and assumption \eqref{assumption}, all solutions of the equation

$$  T_{K_s, h_s}(u)=u$$
are bounded from above. By regularity arguments one obtains that for some $R>0$ all solutions of such equation belong to a certain ball $B_R$ in $X$.

One concludes now by the homotopy invariance of the Leray-Schauder degree, since:

$$ h_1=h, \ h_0 =0;  \ K_1 =K, \ K_0 = \tilde{K}.$$

\end{proof}

%\begin{remark} Let us point out that, in general, it is not possible to pass to via an homotopy as above to a problem with $K=0$, $h>0$. The reason is that, in such case, $\Phi= 2H$. But in such case,
%	$$ d_B(\nabla H, \mathbb{S}^1, 0) \leq 0,$$
%	
%since $H$, being harmonic, admits only critical points of saddle type (each one counting -1 to the degree, in a nondegenerate case). Hence, if $d_B(\nabla \Phi, \mathbb{S}^1, 0) >0$, the problem is not homotopic to a $K=0$ case.
%\end{remark}

Thanks to the above proposition, in the rest of the paper we will be concerned with the Leray-Schauder degree associated to problem \eqref{gg00}. For the sake of clarity in the notation, from this point on we will drop the tilde grapheme and consider the problem

\begin{equation} \label{simple} \left\{\begin{array}{ll}
\displaystyle{-\Delta u = 2Ke^{u} } \qquad & \text{ in }\D, \\
\displaystyle{\frac{\partial u}{\partial \nu}} +2 = 0 \qquad  &\text{ on } \partial \D,
	\end{array}\right.
\end{equation}
for a positive function $K \in C^2(\oD)$. 
It is important to observe, for later use, that any solution of \eqref{simple} satisfies:

\begin{equation} \label{mass} \int_{\D} K e^u = 2 \pi \Rightarrow \frac{2 \pi}{max \, K} \leq \int_{\D} e^u \leq \frac{2 \pi}{min \, K}. \end{equation}

Thanks to Proposition \ref{h=0}, the proof of Theorem \ref{main} reduces to show the following result.

\begin{theorem} \label{main2} Let $K \in C^2(\overline{\D})$, $K>0$.
	Assume that 
	
	\begin{equation} \label{assumption2} \nabla K(x) \neq 0 \ \forall \ x \in \partial \D. \end{equation} 

Then for sufficiently large $R>0$,  

$$deg_{LS}(I-T_{K,0}, B_R, 0)= \pm deg_B (\nabla K, \D, 0).$$
	\end{theorem}

\begin{remark} The $\pm$ sign in the above theorem is due to the possible change of sign when one uses the invariance of the Leray-Schauder degree under homeomorphisms, as will be shown in Section 4. We conjecture that Theorem \ref{main2} holds without change of sign. However, for the purpose of showing existence of solutions, the sign is uninfluential.
\end{remark}

In order to prove Theorem \ref{main2} we will use another homotopy to pass to a problem for which $K$ is arbitrarily close to a positive constant. This idea is reminiscent from \cite{ChGYg, ChYg}. For any $s \in (0,1)$, define 

\begin{equation} \label{Ks} K_s = s K(x) + (1-s), \ s \in (0,1). \end{equation}

\begin{proposition} \label{K casi 1} There exists $R>0$ such that for any $s \in (0,1]$ and any solution $u$ of

\begin{equation} \label{simples} \left\{\begin{array}{ll}
\displaystyle{-\Delta u = 2K_se^{u} } \qquad & \text{ in }\D, \\
\displaystyle{\frac{\partial u}{\partial \nu}} +2 = 0 \qquad  &\text{ on } \partial \D,
\end{array}\right.
\end{equation}
the bound $\|u\| \leq R$ holds. As a consequence,
	$$ deg_{LS} (I-T_{K,0}, B_R, 0) = deg_{LS}(I-T_{K_s,0}, B_R,0).$$

\end{proposition}

\begin{proof} It is worth to point out that in this proposition the radius $R$ is a fixed constant and, in particular, it is independent of $s \in (0,1]$. Via the homotopy property of the degree, as in the previous proposition, we only need to check that the uniform boundedness of the solutions of \eqref{simples}. Observe that this does not come directly from Theorem \ref{blowup}. 

Reasoning by contradiction, let $u_n$ be a blowing-up sequence of solutions of the problem:

$$\left\{\begin{array}{ll}
\displaystyle{-\Delta u_n = 2K_n e^{u_n} } \qquad & \text{ in }\D, \\
\displaystyle{\frac{\partial u_n}{\partial \nu}} +2 = 0 \qquad  &\text{ on } \partial \D.
	\end{array}\right.
	$$
where $K_n= K_{s_n}$, and $s_n \in (0,1)$. In view of Theorem \ref{blowup}, the only problematic case is $s_n \to 0$. In this case, $K_n \to 1$.

By Lemma \ref{estima}, a), we have that $K_n e^{u_n} \rightharpoonup 2 \pi \delta_p$ for some $p \in \partial \D$. Moreover, by \eqref{kw1} we conclude that:

$$ 0 = \frac{1}{s_n} \int_{\D} e^{u_n} \, \partial_{\tau} K_n= \int_{\D} e^{u_n} \, \partial_{\tau} K \to 2 \pi \partial_{\tau} K(p).$$

Assume now, for simplicity, that $p=(1,0)$. By \eqref{kw2} and Lemma \ref{estima}, b), we have that:

$$ 0= \frac{1}{s_n} \int_{\D} e^{u_n} \, \nabla K_n \cdot F = \int_{\D} e^{u_n} \, \nabla K \cdot F = 4 \pi \partial_{\nu}K(p) (1-\lambda_n)+ o(1-\lambda_n).  $$

As a consequence, $\nabla K(p)=0$, which contradicts \eqref{assumption2}.

\end{proof}

By Proposition \ref{K casi 1}, we can focus our attention to problem \eqref{simples}. As we shall see later, this problem has several advantages with respect to \eqref{simple}.

\subsection{The variational formulation}

Let us point out that the right variational formulation for problem \eqref{gg} is a bit tricky (see \cite{CruzRuiz}, where the symmetric case is treated). It could be possible, but rather difficult, to perform the Leray-Schauder degree computation for the original problem \eqref{gg}. However, problem \eqref{simple} is much simpler from a variational point of view, and this allows us to prove Theorem \ref{main2} in a less intricate way.

We now interpret our formulation for the special case \eqref{simple}. Observe that if $h=0$ and $K>0$, the map $c$ defined in \eqref{def-c} is just:

$$ c(u)= \log \left ( \frac{2 \pi}{\int_{\D} K e^u} \right ).$$

Moreover, the compact map $T$ can be rewritten as:

\begin{equation} \begin{array}{c} \label{def-T2} T: X \to X,\ T(u)=v,\ \mbox{ where } v \in X \mbox{ solves:} \\   \\
\left\{\begin{array}{ll} -\Delta v = 4 \pi \frac{K e^u}{\int_{\D} K e^u }  \qquad & \text{ in }\D, \\
\displaystyle{\frac{\partial v}{\partial \nu}} +2 = 0 \qquad  &\text{ on } \partial \D.
\end{array}\right. \end{array}
\end{equation}

We observe that, via the identification of $X$ and $X^*$ by the scalar product,

$$ (I - T)(u) = J'(u),$$

where $J: X \to \R$ is defined as:

\begin{equation} \label{J} J(u) = \frac{1}{2} \int_{\D} |\nabla u |^2 + 2 \int_{\partial \D} u - 4 \pi \log \Big ( \int_{\D} K e^u \Big ). \end{equation}

Let us point out that the $J$ is actually defined in $H^1(\D)$ and it is invariant under addition of constants. We will also define $J_s$ as:

\begin{equation} \label{Js} J_s(u) = \frac{1}{2} \int_{\D} |\nabla u |^2 + 2 \int_{\partial \D} u - 4 \pi \log \Big ( \int_{\D} K_s e^u \Big ), \end{equation}
where $K_s$ is defined in \eqref{Ks}.

Let us recall now the Moser-Trudinger inequality for the Neumann case, see for instance \cite[equation (4)]{osgood}, which implies that $J_s$ is bounded from below.

\begin{proposition}\label{lebedev} The following inequality holds:
	\begin{equation} \label{mt}\log \int_{\D} e^u \leq \frac{1}{8\pi} \int_\D | \nabla u |^2 + \frac{1}{2\pi} \int_{\partial \D} u , \ \forall u\in H^1(\Sigma).\end{equation}
\end{proposition}

However, it is not clear that $J_s$ attains its infimum, since coercivity is lost because of the noncompact effect of the conformal maps of the disk.

\medskip

In the next sections we will study the Leray-Schauder degree of the operator $J_s'$.

\section{The constrained minimization problem}

In order to find critical points of the functional $J_s$ we shall use a constrained minimization approach. For any $a \in \D$, we define:

$$ M_a = \{ u \in X: \int_{\D } e^{u} (x-a)=0\} $$$$= \{ u \in X: \int_{\D } P(u) =a\}, $$
where $P=(P_1,P_2): H^1(\D) \to \R^2$,

\begin{equation} \label{P} P(u)= \frac{\int_{\D} e^u x }{\int_{\D} e^u}. \end{equation}

Observe that also $P$ is invariant under adition of constants.

The boundedness result given in Proposition \ref{K casi 1} implies the following:

\begin{lemma} \label{lemita} There exists $\bar{r} \in (0,1]$ such that for any $s \in (0,1)$ and any solution $u$ of \eqref{simples}, we have $$|P(u)| < \bar{r}.$$
\end{lemma}

\begin{proof} 

By the definition of $P$, we have:

$$ 1 - |P(u)| \geq \frac{\int_{\D} e^u (1-|x|)}{\int_{\D} e^u}.$$

By \eqref{mass}, it suffices to show:

\begin{equation} \label{frombelow} \int_{\D} e^u (1-|x|) \geq \delta >0 \end{equation}

By Proposition \ref{K casi 1}, we have that $\|u\| \leq R$. By inequality \eqref{mt} applied to $p\,  u$, we conclude that $\| e^u \|_{L^p} \leq C$ for any $p>1$. For $r \in (0,1)$ define the disk: $D_r=D(0,r)$. We can estimate, by using Holder inequality,

$$ \int_{D_r} e^u = \int_{\D} e^u  - \ \int_{\D \setminus D_r} e^u \geq \int_{\D} e^u -  \left ( \int_{\D \setminus D_r} e^{2u}   \right )^{1/2} (2 \pi (1-r))^{1/2} $$$$\geq  \int_{\D} e^u  -\ C (1-r)^{1/2}.$$
For a fixed $r \in (0,1)$ sufficiently close to $1$, and using again \eqref{mass}, we can make $ \int_{\D} e^u  -\ C (1-r)^{1/2} > c>0$. Then we conclude

$$ \int_{\D} e^u (1-|x|) \geq \int_{D_r} e^u (1-|x|) \geq c (1-r) >0,$$

which finishes the proof.

\end{proof}

\subsection{The existence of constrained minima} In this subsection we will prove the following proposition.

\begin{proposition} \label{minimo} For any $s \in [0,1]$, $a \in \D$, the functional $ J_s|_{M_a}$ is coercive and hence achieves its minimum.	
\end{proposition}

As commented previously, the functional $J_s$ is bounded from below but is not coercive due to the noncompact action of the conformal group of the disk. However the constraint $P(u)=a$ allows us to recover coercivity and the existence of minimum can be shown as follows. This idea dates back to \cite{Aubin}, and can be seen as a consequence of the so-called Chen-Li type inequalities (see \cite{chen-li} for more details). These are improvements of the inequality \eqref{mt} under assumptions on how $e^u$ is distributed in $\D$. 

\begin{lemma} \label{chen-li}
		Let $\e>0$, $\delta>0$ and $0<\gamma<1/2$. Let $\Sigma_1$, $\Sigma_2$ and $S$ be subsets of $\D$, such that $d(\Sigma_1, \Sigma_2) > \delta$ and $d(S, \partial \D) > \delta$. If
		
		\begin{equation}\label{ChLi1}
		\frac{ \int_{\Sigma_1} e^u}{\int_{\Sigma} e^u } \geq \gamma, \hspace{0.5cm} \frac{ \int_{\Sigma_2} e^u}{\int_{\Sigma} e^u } \geq \gamma,
		\end{equation}
		or
		
		\begin{equation}\label{ChLi2}
		\frac{ \int_{S} e^u}{\int_{\Sigma} e^u } \geq \gamma,
		\end{equation}
		then, there exists a constant $C=C(\e,\delta,\gamma)$, such that for all $u \in X$,
		\begin{equation} \label{improved} \log \int_{\Sigma} e^u \leq  \ \frac{}{(16\pi-\e)} \int_{\Sigma} \left|\nabla u\right|^2 + C. \end{equation}

\end{lemma}

This lemma was first stated in \cite{Wang}; the reader can find a detailed proof in \cite[Lemma 2.4]{lsr}.

\begin{proof}[Proof of Proposition \ref{minimo}]
	
Take $u_n \in M_a$ with $\|u_n\| \to +\infty$. Up to a subsequence, we can assume that
$$ \frac{e^{u_n}}{\int_{\D} e^{u_n}} \rightharpoonup \mu,$$
where $\mu$ is a positive measure, $\mu(\overline{\D})=1$ and the above convergence is the weak convergence of measures. In particular, since $u_n \in M_a$,

\begin{equation} \label{mevoy} \int_{\D} (x-a) \, d \mu =0. \end{equation}

We denote by $\G$ the support of $\mu$, and consider two cases:

\begin{enumerate}
	\item $\G=\{p\} \subset \overline{D}$. In such case, $\mu = \delta_p$. By \eqref{mevoy}, we have necessarily that $p=a$.
	\item $\G$ contains at least two different points $p_1$, $p_2$.
\end{enumerate}
	
In the first case, we define $S\subset \overline{S}\subset \D$ an open set containing $a$. In the second case, we take $p_i \in \Sigma_i \subset \overline{D}$ with $\overline{S}_1 \cap \overline{S}_2 = \emptyset$. By the weak convergence of measures, we have that there exists $\e>0$, $\delta >0$ and $\gamma < 1/2$ such that either \eqref{ChLi1} or \eqref{ChLi2} holds for $u_n$. 

We now recall \eqref{Js} and get that:

$$ J_s(u) \geq \frac{1}{2} \int_{\D} |\nabla u |^2 + 2 \int_{\partial \D} u - 4 \pi \log \Big ( \int_{\D} e^u \Big ) -C.$$
	
Observe now that for any $u \in X$, 

$$ \int_{\partial \D} |u| \leq C \|u\| \leq C + \e \|u\|^2.$$

We now plug estimate \eqref{improved} to the above estimates and conclude that

$$ J_s(u_n) \geq c \int_{\D} |\nabla u_n|^2 - C \to +\infty,$$
which yields coercivity. 

Clearly, the functional $J_s$ is weak lower semicontinuous, and the set $M_a$ is weakly closed. Then the infimum 

$$ inf J_s|_{M_a}$$
is attained.

\end{proof}

\medskip In particular, the minimizer $u_a$ given above is a constrained critical point of $J_s|_{M_a}$, that is, there exists $\tilde{\mu}_i \in \R^2$ such that $J_s'(u_a)= \sum_{i=1}^2 \tilde{\mu}_i P_i'(u_a)$; in other words, for any $v \in H^1(\D)$,

\begin{equation} \label{preprelagrange} \int_{\D} \nabla u_a \cdot \nabla v + 2 \int_{\partial \D} v - 4 \pi \frac{\int_{\D} K_s e^{u_a}v}{\int_{\D} K_s e^{u_a}} = \frac{\int_{\D } e^{u_a} v \  \tilde{\mu} \cdot (x-a)}{\int_{\D} e^{u_a}}. \end{equation}

That is, $u_a$ is a weak solution of the problem:

\begin{equation} \label{prelagrange} \left\{\begin{array}{ll}
\displaystyle{-\Delta u_a = 4 \pi \frac{K e^{u_a}}{\int_{\D} K e^{u_a}} + \frac{e^{u_a} (x-a) \cdot \tilde{\mu} }{\int_{\D}e^{u_a}}} \qquad & \text{ in }\D, \\
\displaystyle{\frac{\partial u_a}{\partial \nu} +2 =  0} \qquad  &\text{ on } \partial \D,
\end{array}\right.
\end{equation}

If we define $w_a=u_a + \log \frac{2 \pi}{\int_{\D} K e^{u_a}}$, then,
	
\begin{equation} \label{lagrange} \left\{\begin{array}{ll}
\displaystyle{-\Delta w_a = 2 \Big (  K +\mu \cdot (x-a) \Big ) e^{w_a} } \qquad & \text{ in }\D, \\
\displaystyle{\frac{\partial w_a}{\partial \nu}} +2 =  0 \qquad  &\text{ on } \partial \D,
\end{array}\right. \ P(w_a) = a.
\end{equation}

with \begin{equation} \label{muu} \mu = \frac{1}{2 \int_{\D} e^{w_a}  } \tilde{\mu}, \ \mbox{ and } \int_{\D} e^{w_a} \in \Big (\frac{2\pi}{max \ K}, \frac{2 \pi}{min \ K} \Big ).\end{equation}

%Let us point out that solutions of \eqref{lagrange} also satisfy \eqref{mass}.

In general, there is no reason to expect that the minimum $min\, J_s|_{M_a}$ is attained at a unique point. We will show next that if $s=0$ the minimizer is unique and explicit. Via a perturbation argument, we will show that this is the case also if $s$ is small.

\subsection{The limit case $s=0$}

\begin{proposition}\label{s=0} For any $a \in \D$ define:
	
	\begin{equation} \label{ua} \varphi_a= \psi_a - \frac{1}{\pi} \int_{\D} \psi_a(z) \, dz,  \ \ \psi_a(z)= 2\log \left(\frac{2 (1-|a|^2)}{|1 - \overline{a} x|^2 + |x-a|^2}\right).  \end{equation}
	
	\begin{enumerate}
		\item[i)] $\varphi_a$ is the unique constrained critical point of $J_0$ in $M_a$,
		\item[ii)] $\varphi_a$ is nondegenerate, in the sense that the kernel of $J_0''(u_a)$ is a two-dimensional space;
		\item[iii)] The set $Z_0= \{\varphi_a:\ a \in \D  \}$ is a smooth $2$-manifold in $X$.
		\item[iv)] $Z_0$ and $M_a$ are transversal at $\varphi_a$, that is, $(T_{\varphi_a}M_a) \oplus (T_{\varphi_a}Z_0) = X$.
		\item[v)] $J_0(\varphi_a)= min \, J_0$.
		
	\end{enumerate}
	
\end{proposition}

\begin{proof} Let $u$ be a constrained critical point of $J_0$ in $M_a$, and denote $w= u + \log \frac{2 \pi}{\int_{\D} e^u }$. By making a suitable rotation we can assume that $ \mu = (\mu_1,0)$. By \eqref{kw2}, we have that:

\begin{equation} \label{prearriba} \mu_{1} \int_{\D} e^{w}   (1-x_1^2+x_2^2) =0.\end{equation}

As a consequence, $\mu_1=0$, and $w$ is a solution of the problem:
	
	\begin{equation} \label{eq: s=0} \left\{\begin{array}{ll}
	\displaystyle{-\Delta w = 2  e^{w} } \qquad & \text{ in }\D, \\
	\displaystyle{\frac{\partial u}{\partial \nu}} +2 = 0 \qquad  &\text{ on } \partial \D.
	\end{array}\right.
	\end{equation}
	
 Those solutions have been classified (see for instance \cite[Lemma 2.1]{jlmr}), and are given by the expression $\psi_{a'}$ in \eqref{ua}. A priori, $a' \in \D$ could be different from $a$, but we now claim that:
	
	$$0 = \int_{\D} e^{{\varphi_a}} (x-a) \, dx = \int_{\D} e^{{\psi_a}} (x-a) \, dx,$$
	
	so that $\varphi_a \in M_a$ and therefore $a'=a$ and $u=\varphi_a$. 
	
	\medskip In order to prove the claim, we observe that:
	
	\begin{equation} \label{dopo} \int_{\D} e^{{\psi_a}} \, dx = 2 \pi.\end{equation}

	We now compute $\int_{\D} e^{{\psi_a}} x \, dx$. Writing $a= a_0 e^{i\theta}$, with $a_0 >0$, $\theta \in \R$, we have:
	
	$$ \int_{\D} e^{{\psi_a}} x \, dx =\int_{\D} \frac{4 (1-a_0^2)^2}{(|1 - a_0 e^{-i \theta} x|^2 + |x-a_0 e^{i\theta}|^2)^2} x \, dx $$$$= e^{i \theta} \int_{\D} \frac{4 (1-a_0^2)^2}{(|1 - a_0 x'|^2 + |x'-a_0|^2)^2} x' \, dx',$$
	by the change of variables $x' = e^{-i \theta} x$ (with abuse of notation, we will write again $x$ for the new variable). By symmetry, it is clear that the expression:
	
	$$\int_{\D} \frac{4 (1-a_0^2)^2}{(|1 - a_0 x|^2 + |x-a_0|^2)^2} x \, dx$$
	is a real number. If we write $x= x_1 + i x_2$ and use polar coordinates, we have:
	
	$$\int_{\D} \frac{4 (1-a_0^2)^2 x_1}{(|1 - a_0 x|^2 + |x-a_0|^2)^2} \, dx $$$$=  \int_0^1\int_0^{2\pi} \frac{4 (1-a_0^2)^2 \, r^2 \cos t}{(1-a_0 r \cos t)^2 + a_0^2 r^2 \sin^2 t + (r \cos t - a_0)^2 + r^2 \sin^2t} \, dr \,  dt $$ $$ = 4(1-a_0^2)^2 \int_0^1 \int_0^{2\pi} \frac{r^2 \cos t}{(1+a_0^2 + r^2 + a_0^2 r^2 - 4 a_0 r \cos t)^2} \, dr \, dt. $$
	An elementary but lengthy computation (we used the software Mathematica here) shows that the above espression is equal to $ 2 \pi a_0$.
	Then, 
	
	$$ \int_{\D} \frac{4 (1-|a|^2)^2}{(|1 - \overline{a} x|^2 + |x-a|^2)^2} z \, dz = e^{i \theta} 2 \pi a_0 = 2 \pi a.$$ 
	
	This, together with \eqref{dopo}, finishes the proof of i).
	
	\medskip We now prove ii). In order to do that, we will characterize the solutions $\phi \in H^1(\D)$ of the linearized problem:
	
	$$ \left\{\begin{array}{ll}
	\displaystyle{-\Delta \phi = 4 \pi \frac{e^{\varphi_a} \phi \int_{\D}e^{\varphi_a}  - e^{\varphi_a} \int_{\D}e^{\varphi_a} \phi }{(\int_{\D} e^{\varphi_a})^2}} \qquad & \text{ in }\D, \\
	\displaystyle{\frac{\partial \phi}{\partial \nu}} =  0 \qquad  &\text{ on } \partial \D.
	\end{array}\right.
	$$
	
	By the invariance of additive constants, we can replace $\varphi_a$ by $\psi_a$ in the equation above, that reads:
	
	$$ \left\{\begin{array}{ll}
	\displaystyle{-\Delta \phi = 2e^{\psi_a} \phi - \frac{e^{\psi_a}}{\pi } \int_{\D} e^{\psi_a}} \phi \qquad & \text{ in }\D, \\
	\displaystyle{\frac{\partial \phi}{\partial \nu}} = 0 \qquad  &\text{ on } \partial \D.
	\end{array}\right.
	$$
	
	Let us point out that $\phi=1$ is a solution (not belonging to $X$) of the above equation. Given any nonconstant solution $\phi \in H^1(\D)$ we can assume, by substracting a suitable constant, that 
	
	$$ \int_{\D} e^{\psi_a} \phi =0.$$
	
	As a consequence, $\phi$ solves the equation:
	
	$$ \left\{\begin{array}{ll}
	-\Delta \phi = 2e^{\psi_a} \phi \qquad & \text{ in }\D, \\
	\displaystyle{\frac{\partial \phi}{\partial \nu}} = 0 \qquad  &\text{ on } \partial \D.
	\end{array}\right.
	$$
	
	The solutions of the above equation have been found in \cite[Lemma 2.3]{jlmr}, and they form a two dimensional space. Indeed in such lemma only the case $a=0$ is stated, but the same argument works for any $a \in \D$.
	
	In sum, the kernel $J_0''(u_a)$ in $H^1(\D)$ has dimension $3$, one of its generators being the constant function $1$. Then, when restricted to $X$, the kernel has dimension 2.

	\medskip The proof of iii) is immediate. We now prove iv) by showing that $(T_{\varphi_a}M_a) \cap (T_{\varphi_a}Z_0) = \{0\}$. Observe that $v \in T_{\varphi_a}M_a$ if and only if:
	
	$$P'(\varphi_a)(v)=0 \Rightarrow \int_{\D} e^{\varphi_a} (x_i-a_i) v=0, \ i=1,\ 2.$$
	
	However,
	
	$$ \int_{\D} e^{\varphi_a} (x_i-a_i) \partial_{a_i} \varphi_a =  \partial_{a_i} \underbrace{\left ( \int_{\D}  e^{\varphi_a} (x_i-a_i) \right )}_{=0} + \int_{\D} e^{\varphi_a} = 0 + 2\pi \neq 0.$$

	We now prove v). As shown by Proposition \ref{minimo}, $J_0$ achieves its infimum in $M_a$. Since $\varphi_a$ is the unique constrained critical point of $J_0$ in $M_a$, then $J_0(u_a)= min\ J_0|_{M_a}$. Even more, since $Z_0$ is a smooth manifold of critical points, we have that $J_0(u_a)$ is constant in $a$, that is, $J_0(u_a) = min\, J_0$ for all $a \in \D$.

\end{proof}

\subsection{The case $s$ small}

In this subsection we will study problem \eqref{simples} for a small value of $s$, as a perturbation of the case $s=0$. There are some technical issues to do this, though. The basic tool is the uniform bound given in Proposition \ref{K casi 1}.

The main result of this subsection is the following.

\begin{proposition} \label{s(r)} For any $r \in (0,1)$ there exists $s(r)$ such that, for any $s \in (0,s(r))$ and any  $a \in \D$ with $|a| <r$, the following assertions hold true:
	
\begin{enumerate}
	
	\item[a)] The minimizer provided by Proposition \ref{minimo} is unique; we shall denote it by $u_a=u_a^s$, and we have $u_a^s \to \varphi_a$ as $s \to 0$ in $H^1$ sense.
	
	\item[b)] The minimizer $u_a$ is nondegenerate; being more specific, $J_s''(u_a)$ has no kernel and gives rise to a positive definite quadratic form in $T_{u_a}M_a$. As a consequence, the set of minimizers $Z_s(r)=\{ u_a;\ a \in \D, |a|<r,\ \}$ is a smooth 2-manifold. 
	
	\item[c)] The function $w_a= u_a + \log \frac{2\pi}{\int_{\D} K e^{u_a}}$ solves equation \eqref{lagrange} with some $\mu_s(a) \in \R^2$ which is smooth in $a$. Moreover, $|\mu_s(a)|=o_s(1)$ uniformly for $|a| \leq r$.
	
	\item[d)] All solutions of \eqref{simples} belong to $Z_s(\bar{r})$, where $\bar{r}$ is given by Lemma \ref{lemita}.

\end{enumerate}
\end{proposition}

\begin{proof}
	
We first prove a) by contradiction; fix $r \in (0,1)$, $s_n \to 0$ and denote by $u_n$, $v_n$ two different constrained critical points of $J_{s_n}|_{M_{a_n}}$, with $|a_n| \leq r$. As in \eqref{lagrange}, defining $w_n=u_n + \log \frac{2 \pi}{\int_{\D} K_{s_n} e^{u_n}}$, we have

\begin{equation} \left\{\begin{array}{ll}
\displaystyle{-\Delta w_n = 2 \Big (  K_{s_n} +\mu_n \cdot (x-a_n) \Big ) e^{w_n} } \qquad & \text{ in }\D, \\
\displaystyle{\frac{\partial w_n}{\partial \nu}} +2 =  0 \qquad  &\text{ on } \partial \D,
\end{array}\right. \ P(w_{a_n}) = a_n.
\end{equation}

Observe also that $K_{s_n} \to 1$, so that

\begin{equation} \label{masa} \int_{\D} e^{w_n} \to 2 \pi.\end{equation}

We first claim that $|\mu_n|$ is bounded. If this claim is true, we can make use of \cite[Proposition 3.1]{jlmr}. Let us point out that such proposition is applicable by \eqref{masa} since, up to a subsequence, we can assume that $K_n= K_{s_n} +\mu_n \cdot (x-a_n)$ converges in $C^2$ sense (even if we do not know anything about its sign). Since $|P(u_n)| \leq r<1$, $u_n$ cannot be a blowing-up sequence, and hence it is uniformly bounded. Passing to a subsequence and using standard estimates we get that $u_n \to \varphi_{a}$, where $a_n \to a$.

We can apply the same argument to the sequence $v_n$, obtaining that $v_n \to \varphi_a$. By Proposition \ref{s=0}, ii),we can use the Implicit Function Theorem to conclude uniqueness. Observe that we also conclude that $u_a^s \to \varphi_a$ as $s \to 0$.

We now prove the claim. Reasoning by contradiction, assume that $|\mu_n| \to +\infty.$ By making a suitable rotation, we can assume that $\mu_n(a) = (\mu_{1,n}, 0)$, with $\mu_{1,n} \to +\infty$. We now use \eqref{kw2} and obtain:

\begin{equation} \label{arriba} \int_{\D} e^{w_n} (\partial_{x_1} K_{s_n} + \mu_{1,n} )(1-x_1^2+x_2^2) - 2 \int_{\D} e^{w_n} x_1 \, x_2 \, \partial_{x_2} K_{s_n}=0.\end{equation}

Clearly, by \eqref{mass},

\begin{equation} \label{one} \int_{\D} e^{w_n} |\partial_{x_1} K_{s_n} | (1-x_1^2+x_2^2) \leq C s_n, \end{equation}

\begin{equation} \label{two} \int_{\D} e^{w_n} |x_1| \, |x_2| \, |\partial_{x_1} K_{s_n}| \leq C s_n. \end{equation}

As a consequence,

$$ \int_{\D} e^{w_n} (1-x_1^2+x_2^2) \to 0.$$ 
The point here is that the function $(1-x_1^2+x_2^2)$ is positive and vanishes only at the points $(\pm 1, 0)$; this implies that:

$$ e^{w_n} \weakto \alpha \delta_p + \beta \delta_q,$$

where $p=(1,0)$, $q=(-1,0)$ and $\alpha$, $\beta$ are positive contants such that
$ \alpha + \beta = 2 \pi$. In particular, $\int_{\D} e^{2w_n} \to +\infty$.
Moreover, by \eqref{masa} and the Jensen inequality,

\begin{equation} \label{average} \limsup \int_{D} w_n \leq \pi \log 2 \end{equation}

By applying the Moser-Trudinger inequality \eqref{mt} to $2w_n$, we conclude that 

$$ \int_{\D} |\nabla u_n|^2  = \int_{\D} |\nabla w_n|^2 \to + \infty.$$  

By Proposition \ref{minimo}, $J_s(u_n) \to +\infty$ and then $u_n$ cannot be a minimizer. The claim is proved.

\medskip 

The proof of b), c) follows from Proposition \ref{s=0}, ii) and iv), as well as the Implicit Function Theorem.

\medskip We finally prove d). Take $u_n$ a solution of \eqref{simples} for $s_n \to 0$. By Proposition \ref{K casi 1}, $u_n$ is uniformly bounded. By compactness, it converges to a solution for $s=0$, and hence it is equal to a function $\varphi_a$ given in Proposition \eqref{s=0}. By Lemma \ref{lemita}, $|a| \leq \bar{r}$. By the Implicit Function Theorem, $u_n$ coincides with some $u_{a_n}^{s_n}$. 

%Define now:
%$\xi_n = \frac{1}{\mu_{1,n}} w_n$, which solves the equation:
%
%\begin{equation} \left\{\begin{array}{ll}
%\displaystyle{-\Delta \xi_n = 2 \Big (  \frac{K_{s_n}}{\mu_{1,n}} + (x_1-a_{1,n}) \Big ) e^{w_n} } \qquad & \text{ in }\D, \\
%\displaystyle{\frac{\partial \xi_n}{\partial \nu}} +\frac{2}{\mu_{1,n}} =  0 \qquad  &\text{ on } \partial \D,
%\end{array}\right. 
%\end{equation}
%
%By the above computations,
%
%$$ 2 \Big (  \frac{K_{s_n}}{\mu_{1,n}} + (x_1-a_{1,n}) \Big ) e^{w_n} \weakto 2 \pi (1-a_1^2) (\delta_p - \delta_q).$$
%
%Moreover, by \eqref{masa} and the Jensen inequality,
%
%\begin{equation} \label{average} \limsup \int_{D} w_n \leq \pi \log 2 \Rightarrow \limsup \int_{D} \xi_n \leq 0. \end{equation}
%
%Hence, $\xi_n - \frac{1}{\pi} \int_{D} \xi_n$ converges in $W^{1,p}$ sense, ($p<2$) to $2 \pi (1-a_1^1) G(x)$, where  
%
%$$ G(x)= \frac{1}{\pi} \log \left (\frac{|x+1|}{|x-1|} \right ) =  \frac{1}{2\pi} \log \left (\frac{(x+1)^2+y^2}{(x-1)^2+y^2} \right ).$$ 
%

\end{proof}

\begin{remark} \label{remark extra} Without loss of generality, we can take the value $s(r)$ given in the previous proposition such that $s(r)$ converges to $0 $ as $r \to 1$. \end{remark}

Once we know that the minimizer is unique, the Lagrange multiplier $\mu_s(a)$ is well defined and continuous. Moreover, as $s$ tends to $0$, also the multiplier tends to $0$. Next proposition gives a sharper asymptotic estimate of the multiplier for small $s$ and $a$ close to $\partial \D$. This will be the key to relate the Leray-Schauder degree of our problem with the Brouwer degree of $\nabla K$.

\begin{proposition} \label{chiave} For any $\e>0$ there exists $r_\e \in (0,1)$ such that the following holds true:
	
\medskip 
For any $r \in (r_\e, 1)$, $s \in (0, s(r))$ (as defined in Proposition \ref{s(r)}) and any $a \in \D$ with $|a| = r$, there holds:
	$$ \left | \frac{\mu_s(a)}{s} + \nabla K(a) \right | < \e,$$
where $\mu_s(a)$ is given in Proposition \ref{s(r)}, c).

\end{proposition}

\begin{proof} Take $r_n \in (0,1)$, $r_n \to 1$, $s_n \in (0, s_n(r))$,  $a_n \in \D$, $|a_n| = r_n$, and denote by $u_n= u_{a_n}$, $w_n = u_n + \log \frac{2 \pi}{\int_{\D} e^{u_n}}$. Clearly $w_n$ is a solution of the problem 
	
\begin{equation} \label{lagrangen} \left\{\begin{array}{ll}
-\Delta w_n =  2  K_n e^{w_n} \qquad & \text{ in }\D, \\
\displaystyle{\frac{\partial w_n}{\partial \nu}} +2 =  0 \qquad  &\text{ on } \partial \D,
\end{array}\right. ,
\end{equation}
where $K_n(x)=  K_{s_n}(x) + \mu_n \cdot (x-a_n)$.  Observe that by Proposition \ref{s(r)}, c), $K_n(x)> \frac{1}{2}$ for sufficiently large $n$. The proposition is proved if we show that:

\begin{equation} \label{objetivo}  \frac{\mu_n}{s_n} + \nabla K(a_n) \to 0. \end{equation}

Observe that since $|P(w_n)| = |a_n|  \to 1$, $w_n$ is a blowing-up sequence of solutions. Passing to a subsequence, we can assume that $a_n \to p \in \partial \D$. By Remark \ref{remark extra} and Proposition \ref{s(r)}, c), we conclude that $s_n \to 0$ (so that $K_n \to 1$) and $\mu_n \to 0$.

Let us denote $\sigma_n = max \{ s_n, |\mu_n|  \}$. Up to a subsequence we can assume that:

$$ \frac{s_n}{\sigma_n}  \to \tilde{s}_0, \ \ \frac{\mu_n}{\sigma_n}  \to \tilde{\mu}_0,$$

where either $\tilde{s}_0=1$ or $|\tilde{\mu}_0|=1 $ (or both). 

By Theorem \ref{blowup}, we have that $e^{u_n} \rightharpoonup 2 \pi \delta_p$. Moreover, by \eqref{kw1}  we conclude that:

$$ 0 = \frac{1}{\sigma_n} \int_{\D} e^{u_n} \, \partial_{\tau} K_n= \frac{s_n}{\sigma_n}\int_{\D} e^{u_n} \, \partial_{\tau} K +  \int_{\D} e^{u_n} \, \frac{\mu_n \cdot \tau}{\sigma_n}  $$\begin{equation} \label{y uno} \Rightarrow \tilde{s}_0 \, \partial_{\tau} K(p) +  \tilde{\mu}_0 \cdot \tau(p)=0.  \end{equation}

Assume now, for simplicity, that $p=(1,0)$. By \eqref{kw2} we have that: 

$$ 0= \frac{1}{\sigma_n} \int_{\D} e^{u_n} \, \nabla K_n \cdot F =  \frac{s_n}{\sigma_n} \int_{\D} e^{u_n} \, \nabla K \cdot F +  \int_{\D} e^{u_n} \, \frac{1}{\sigma_n} \nabla (\mu_n (x-a_n)) \cdot F. $$

Taking now into account Lemma \ref{estima}, b), we conclude that:

 \begin{equation} \label{y dos} \tilde{s}_0 \partial_{\nu}K(p) +  \tilde{\mu}_0 \cdot \nu(p)=0.  \end{equation}

By \eqref{y uno}, \eqref{y dos}, we conclude that both $\tilde{s_0} \neq 0$, $\tilde{\mu}_0 \neq 0$, and

$$  \nabla K(p) + \frac{1}{\tilde{s}_0} \tilde{\mu}_0=0.$$
 
It suffices to observe that $ \frac{1}{s_n} \mu_n \to \frac{1}{\tilde{s}_0} \tilde{\mu}_0$ to conclude the proof of \eqref{objetivo}.

\end{proof}

\section{Proof of Theorem \ref{main2}} Observe that the Proposition \ref{chiave} readily implies that the map $a \mapsto \mu_s(a)$ has a Brouwer degree which concides with that of $\nabla K$. This already implies the existence of a solution for problem \eqref{simples} for small $s$. However, in order to prove Theorem \ref{main2}, and taking into account Proposition \ref{K casi 1}, we need to compute the associated Leray-Schauder degree. 

Then, more work is needed. A computation in the same spirit is given in \cite[pp. 177-180]{Ji}, based in the previous work \cite{preJi}. Here we give an alternative proof.

\begin{definition} Take $r^* \in (\bar{r},1)$ (where $\bar{r}$ is given in Lemma \ref{lemita}) so that Proposition \ref{chiave} holds for 
$$\e < \min \{ |\nabla K(p)|;\ p \in \D, |p| \in [r^*, 1]\}.$$	

We take $s \in (0, s(r^*))$ as given in Proposition \ref{s(r)}. We write as in Proposition \ref{s(r)},
	
	$$ Z_s = \{ u_a: a \in \D, \ |a|<r^*\}.$$
	
	We also fix $\delta>0$ and define:
	
	$$ Z_s^\delta =\{u \in X: u = u_a + \tilde{u},\ |a| < r^*, \ \|\tilde{u}\| < \delta, \ \tilde{u} \in \ T_{u_a}M_a  \}.$$
	
	Taking into account Proposition \ref{s=0}, iv), $Z_s^\delta$ is a tubular neighborhood of $Z_s$.
	
\medskip We denote by $\pi_a$ and $\pi_a^{\perp}$ the orthogonal projections onto $T_{u_a}M_a$ and  $\left(T_{u_a}M_a \right)^{\perp}$, respectively. Let us recall that

\begin{equation} \label{tangente} T_{u_a}M_a= \{ v \in X: \ \int_{\D} e^{u_a} v (x-a)=0\}. \end{equation}

 %In particular,
%	$$\pi_a( u_a + \tilde{u})= \tilde{u}.$$
%	

\end{definition}

\begin{proof}[Proof of Theorem \ref{main2}]

Thanks to Proposition \ref{s(r)}, d), the excision property of the degree allows us to study:

$$ deg_{LS}(J_s', Z_s^\delta, 0).$$

By Proposition \ref{s(r)}, b), we have that

  \begin{equation} \label{def+} J_s''(u_a) \mbox{ is strictly positive definite in }T_{u_a}M_a.\end{equation}
As a consequence, for a sufficiently small $\delta>0$,

$$ \langle J_s'(u_a+\tilde{u}) , \tilde{u} \rangle >0 \mbox{ if } \tilde{u} \in T_{u_a}M_a, \ \| \tilde{u}\| \in (0, \delta).$$
Then, the operator: 
$$ D: B(0, \delta) \subset T_{u_a}M_a \to T_{u_a}M_a,$$

$$D(\tilde{u}) = \pi_a \left( J_s'(u_a+ \tilde{u}) \right),$$
is homotopic to the identity. Hence we can use the homotopy property of the degree, to conclude that:

$$ deg_{LS}(J_s', Z_s^\delta, 0)=deg_{LS}(F_1, Z_s^\delta, 0),$$
where $F_1: Z_s^\delta \to X$ is defined as

$$ F_1(u_a + \tilde{u}) = \pi_a^{\perp} \left( J_s'(u_a+ \tilde{u})\right) + \tilde{u}.$$

Define now $H(t, u_a + \tilde{u}) = t F_1(u_a + \tilde{u}) + (1-t) F_2(u_a + \tilde{u})$, where $F_2: Z_s^\delta \to X$ is defined as

$$ F_2(u_a + \tilde{u}) = J_s'(u_a) + \tilde{u}.$$

We show that $H$ is a valid homotopy. Indeed, 

$$0= H(t, u_a + \tilde{u})= t \pi_a^{\perp} \left( J_s'(u_a+ \tilde{u})\right) + (1-t) J_s'(u_a) + \tilde{u}.$$

Observe now that the two first terms belong to $T_{u_a}M_a$ and $\tilde{u} \in (T_{u_a} M_a)^{\perp}$, hence we have:

$$ t \pi_a^{\perp} \left( J_s'(u_a+ \tilde{u})\right) + (1-t) J_s'(u_a)=0, \  \tilde{u}=0.$$

But if $\tilde{u} =0$ then $J_s'(u_a)=0$, and this does not occur in $\partial Z_s^\delta$ by Proposition \ref{chiave}. By the homotopy invariance of the Leray-Schauder degree we have that:

$$ deg_{LS}(F_1, Z_s^\delta, 0)=deg_{LS}(F_2, Z_s^\delta, 0).$$

In order to compute the Leray-Schauder degree of $F_2$, we will use a convenient homeomorphism. We make the abuse of notation $x_i \in X$, $x_i(x)=x_i$, $i=1$, $2$, and $x=(x_1, x_2)$. We write $E= span \{x_1, \ x_2\} \subset X$, and 

$$D(r^*)= \{a \cdot x,\ a=(a_1, a_2) \in \R^2, \ \| a \| <r^*\} \subset E, $$$$ \ B_\e=B(0, \e) \subset E^{\perp} \subset X,$$ for sufficiently small $\e>0$. We claim that for any $a \in D(r^*)$, $E$ and $T_{u_a}M_a$ are transversal, that is,

\begin{equation} \label{claimfinal} E \cap T_{u_a}M_a = \{0\}= E^{\perp} \cap (T_{u_a}M_a)^{\perp}.\end{equation} 

In order to prove this, observe that since $P(u_a)=a$, we have that:

$$ \int_{\D} a_1 (x_1-a_1) e^{u_a} = a_1 \int_{\D} (x_1-a_1) e^{u_a} =0.$$

Then

\begin{equation} \label{positivo} \int_{\D} x_1 (x_1-a_1) e^{u_a} =  \int_{\D} (x_1-a_1)^2 e^{u_a} >0,\end{equation}

which, recalling \eqref{tangente}, implies that $x_1 \notin T_{u_a}M_a$. In this way we conclude that $E \cap T_{u_a}M_a = \{0\}$. The second assertion follows from the fact that $\dim \, E= \dim \, (T_{u_a}M_a)^{\perp} =2$.

\medskip Let us define:

$$ \Gamma: D(r^*) \oplus B_\e \to Z_s^{\delta},$$

$$ \Gamma((a \cdot x) + y)= u_a + \pi_a(y).$$

Observe that by \eqref{claimfinal}, $\pi_a(y)=0 $ only if $y=0$. As a consequence, $\pi_a: B_\e \to T_{u_a}M_a$ is a bijection. We also point out that $\Gamma$ is a compact perturbation of the identity map.

%Observe that for any $a \in D(0,r)$, $e^{-u_a} y -\frac{1}{\pi} \int_{\D} e^{-u_a} y  \in T_{u_a}M_a$. 
%
%Since $0 < c \leq e^{-u_a} \leq C$ for all $|a|\leq r$, we have that $\Gamma$ is a generalized LS map, see \cite[Chapter 3, Section 6]{Rothe}. 

Moreover, $\Gamma$ is an homeomorphism onto its image, which is included in $Z_s^{\delta}$ for $\e$ sufficiently small. Recall that the zeroes of $F_2$ are contained in $Z_s$, so we can use the excision property of the degree to restrict ourselves to the image of $\Gamma$. 

By the invariance of the degree via homeomorphisms (see \cite[Corollary 2.5 of Chapter 5]{Rothe}) we have that:

$$ deg_{LS}(F_2, Z_s^\delta, 0) = \pm deg_{LS}(F_2\circ \Gamma, D(r^*) \oplus B_\e, 0),$$
where
$$ F_2\circ \Gamma ((a \cdot x) + y) = J_s'(u_a) + \pi_a(y).$$

Let us define $ F_3, F_4 : D(r^*) \oplus B_\e \to X$,

$$ F_3 ((a \cdot x) + y) =  J_s'(u_a)  + y.$$
$$ F_4 ((a \cdot x) + y) = \pi_{E}\big ( J_s'(u_a) \big ) + y.$$

Now we define the homotopies $$ \bar{H}(t, w)= t (F_2 \circ \Gamma)(w) + (1-t) F_3(w),$$ $$\tilde{H}(t,w)= t F_3(w) + (1-t)F_4(w).$$ 

Let us check that $\bar{H}$ is a valid homotopy. Assume that for some $t \in [0,1]$, $w = ((a \cdot x) + y) \in D(r^*) \oplus B_\e$, we have $\bar{H}(t,w)=0$, that is,

$$ J_s'(u_a) + t \pi_a (y) + (1-t) y=0.$$ 

If we make a scalar product with $\pi_a(y)$ we obtain that:

$$ t \| \pi_a(y) \|^2 + (1-t) \langle y, \pi_a(y) \rangle=0.$$
Taking into account that $\langle y, \pi_a(y) \rangle = \| \pi_a(y) \|^2$, we conclude that $\pi_a(y)=0$. By \eqref{claimfinal}, we conclude that $y=0$. This would imply that $J_s'(u_a)=0$, but this does not happen in $\partial D(r^*)$, as we know by Proposition \ref{chiave}.

Similar computations imply that $\tilde{H}$ is also a valid homotopy, by multiplying this time  by $\pi_E(J_s'(u_a))$.

%$$F_3:  D(r^*) \oplus B_\e \to X, \ F_3(a \cdot x \oplus y)= \mu(a) \cdot x \oplus e^{u_a} y + e^{u_a} \frac{\int_{\D} e^{u_a} y}{\int_{\D} e^{u_a}}.$$
%
%No, creo que es:

%$$F_3:  D(r^*) \oplus B_\e \to X, \ F_3(a \cdot x \oplus y)= \mu(a) \cdot x \oplus \pi_a(y) ,$$
%
%$$F_4:  D(r^*) \oplus B_\e \to X, \ F_4(a \cdot x \oplus y)= \mu(a) \cdot x \oplus y.$$
%
%Define the homotopies $H(t)= t F_2 \circ \Gamma + (1-t) F_3$, and $\tilde{H}(t)= t F_3 + (1-t) F_4$. Thanks to claim \eqref{claimfinal}, $H$ is a valid homotopy. Alsi $\tilde{H}$ is clearly valid homotopy, since 
%
%$$ \int_{\D} \Big ( e^{-u_a} y - \frac{1}{\pi} \int_{\D} e^{-u_a} y \Big ) y >0,$$
%
%if $\|y\|= \e>0$.
%
%Hence we conclude that:
%
%$$ deg_{LS}(F_2 \circ \Gamma, D(r^*) \oplus B_\e, 0) = deg_{LS}(F_4, D(r^*) \oplus B_\e, 0).$$

Now observe that by \cite[Theorem 8.7]{deimling}, we have that:

$$ deg_{LS}(F_4, D(r^*) \oplus B_\e, 0) = d_B(\Upsilon, D(r^*),0),$$
where $\Upsilon: D(r^*) \to E$, $\Upsilon(a \cdot x) = \pi_E(J_s'(u_a)).$ Recall that

$$ J_s'(u_a)(v)= \frac{1}{\int_{\D} e^{u_a} }\int_{\D} e^{u_a} v \ \tilde{\mu}_{s}(a) \cdot (x-a),$$
and that $\tilde{\mu}_s $ is related to $\mu_s$ via equation \eqref{muu}. 

To conclude the proof of Theorem \ref{main2}, it suffices to prove that:

$$d_B(\Upsilon, D(r^*),0) = d_B (\mu_s, D(r^*), 0).$$ 

The rest of the proof is devoted to verify the above identity. Let us fix $a \in D(r^*)$ and use a coordinate system so that $a=(a_1,0)$. 
If we denote by $\rho = J_s'(u_a)$ via the identification between $X$ and $X'$,
we have that $\rho= \rho_1 + \rho_2$, where $\rho_i$ are weak solutions to the problem:

$$ - \Delta \rho_i = \frac{1}{\int_{\D} e^{u_a} } \tilde{\mu}_{i,s}(a) e^{u_a} (x_i-a_i).$$

Observe now that,
$$ \pi_E(\rho) = \frac{1}{\pi} \langle \rho, x_1 \rangle x_1 + \frac{1}{\pi} \langle \rho, x_2 \rangle x_2,$$
and, taking into account \eqref{positivo},
$$ \frac{1}{ \pi } \langle \rho_i, x_i \rangle = \frac{1}{ \pi } \frac{1}{\int_{\D} e^{u_a} } \tilde{\mu}_{i,s}(a) \int_{\D} e^{u_a} (x_i-a_i) x_i $$ 
$$ = \tilde{\mu}_{i, s}(a) \, q_{i,s}(a)>0, $$ 
where $q_{i,s}(a)> q>0$ for all $|a|\leq r^*$, $s \in (0, s(r^*))$. Observe moreover that:

$$ \frac{1}{ \pi } \langle \rho_i, x_j \rangle = \frac{1}{ \pi } \frac{1}{\int_{\D} e^{u_a} } \tilde{\mu}_{i,s}(a) \int_{\D} e^{u_a} (x_1-a_1) x_2,$$
and 

$$ \int_{\D} e^{u_a} (x_1-a_1) x_2=o(s),$$ 
since, by symmetry,

$$ \int_{\D} e^{\varphi_a} (x_1-a_1) x_2=0,$$
where $\varphi_a$ is defined in \eqref{ua}. 

Then,
$$ \Upsilon (a \cdot x) = q_1(a) \tilde{\mu}_{1,s}(a) x_1 +  q_2(a) \tilde{\mu}_{2,s}(a) x_2 + o(s) (|\tilde{\mu}_s|).$$  
Taking smaller $s$ if necessary, to conclude that
$$ d_B(\Upsilon, D(r^*),0) = d_B(\bar{\Upsilon}, D(r^*),0),$$
where

$$ \bar{\Upsilon} (a \cdot x) = q_1(a) \tilde{\mu}_{1,s}(a) x_1 +  q_2(a) \tilde{\mu}_{2,s}(a) x_2.$$   

We now use a homotopy passing from $q_i(a)$ to $1$ to conclude that:

$$d_B(\bar{\Upsilon}, D(r^*),0) = d_B (\tilde{\mu}_s, D(r^*), 0)= d_B (\mu_s, D(r^*), 0).$$ 

Hence the proof is concluded by taking into account Proposition \ref{chiave}.

\end{proof}

\bibliographystyle{unsrt}

\begin{thebibliography}{99}

\bibitem{Aubin}{T. Aubin, }{Meilleures constantes dans le th\'{e}or\`{e}me d'inclusion de Sobolev et un th\'{e}or\`{e}me de Fredholm non lin\'{e}aire pour la transformation conforme de la courbure scalaire, }{ J. Functional Analysis 32 (1979), no. 2, 148-174.}

%\bibitem{aub}{T. Aubin,}{ { Some nonlinear problems in Riemannian geometry},}{ Springer Monographs in Mathematics. Springer-Verlag, Berlin, 1998.} 

%\bibitem{bao}{J. Bao, L. Wang, C. Zhou, }{ {Blow-up analysis for solutions to Neumann boundary value problem}, }{Journal of Math. Analysis and Appl. 418 (2014), 142--162.}

%\bibitem{bclt}{D. Bartolucci, C-C. Chen, C-S. Lin, G. Tarantello, }{ {Profile of blow-up solutions to mean field equations with singular data}, }{Comm. Partial Differential Equations 29 (2004), no. 7-8, 1241--1265. }
%
%
%\bibitem{bjly}{D. Bartolucci, A. Jevnikar, Y. Lee, W. Yang, }{ {Uniqueness of bubbling solutions of mean field equations}, }{J. Math. Pures Appl., 123 (2019), 78--126.}

\bibitem{angela}{L. Battaglia, M. Medina, A. Pistoia, }{Large conformal metrics with prescribed Gaussian and geodesic curvatures, }{Calc. Var. Partial Differential Equations 60 (2021), no. 1, Paper No. 39.}

%\bibitem{bls}{L. Battaglia, R. L\'{o}pez-Soriano, } {A double mean field related to a curvature prescription problem, }{J. Differential Equations, (2020) \href{https://doi.org/10.1016/j.jde.2020.02.012}{https://doi.org/10.1016/j.jde.2020.02.012}}.


%\bibitem{BJMR}{L. Battaglia, A. Jevnikar, A. Malchiodi, D. Ruiz, } {A general existence result for the Toda system on compact surfaces, }{Adv. Math. 285 (2015), 937--979.}

\bibitem{ber}
{M. Berger,}{  {On Riemannian structures of prescribed Gaussian curvature for compact 2-manifolds},}{ J. Diff. Geom. 5 (1971), 325--332.} 

\bibitem{brendle}{S. Brendle, }{ { A family of curvature flows on surfaces with boundary}, }{Math. Z. 241 (2002), no. 4, 829--869.}

\bibitem{breme}{H. Brezis, F. Merle, } {Uniform estimates and blow-up behavior for solutions of $-\Delta u =V(x) e^u$ in two dimensions, }{Commun. Partial Differential Equations 16 (1991), 1223--1253.}

\bibitem{ChGYg}{S.Y.A. Chang, M. Gursky, P.C. Yang, }{ {The scalar curvature equation on 2- and 3-spheres}, }{Calc. Var. Partial Differential Equations 1 (1993), no. 2, 205--229.}


\bibitem{ChL}{K. C. Chang, J. Q. Liu, }{On Nirenberg problem, International J. Math. 4 (1993), 35-58.}

\bibitem{kcchang}{K.C. Chang, J.Q. Liu, } {A prescribing geodesic curvature problem, }{Math. Z. 223 (1996), 343-365.} 


	\bibitem{ChYgActa}{S.Y.A. Chang, P.C. Yang, } {Prescribing Gaussian curvature on $\mathbb{S}^2$, }{Acta Math. 159 (1987), no. 3--4, 215--259.}
	
	\bibitem{ChYg}{S.Y.A. Chang, P.C. Yang, }{ {Conformal deformation of metrics on $\mathbb{S}^2$}, }{J. Diff. Geom. 27 (1988), 259--296.}
	
	\bibitem{ChYg2}{S. Y. A. Chang, P.C. Yang, }{A perturbation result in prescribing scalar curvature on $\mathbb{S}^2$,}{ Duke Math. J. 64, 27-69 (1991)}
	
	\bibitem{chen-ding}{W. Chen, W. Ding, }{Scalar curvatures on $\mathbb{S}^2$, }{TAMS Vol. 303, Number 1 (1987), 365-382.}
	
	\bibitem{chen-ding2}{W. Chen, W. Ding, }{A problem concerning the scalar curvature on $\mathbb{S}^2$, }{Kexue Tongbao (English Ed.) 33 (1988), no. 7, 533-537. }

%\bibitem{C-L}{W. X. Chen, C. Li, C. (1991), }{Classification of solutions of some nonlinear elliptic
%equations, }{Duke Math. J. 63(3):615–622.}

	
	\bibitem{CL-CPAM}{W. Chen, C. Li, }{A necessary and sufficient condition for the Nirenberg problem. (English summary), }{ 
		Comm. Pure Appl. Math. 48 (1995), no. 6, 657-667. }

\bibitem{chen-li}{W.X. Chen, C. Li, }{{Prescribing Gaussian curvatures on surfaces with conical singularities}, }{J. Geom. Anal. 1-4 (1991) pp. 359-372}.


%\bibitem{ChLin}{C.C. Chen, C.S. Lin, }{ {Sharp estimates for solutions of multi-bubbles in compact Riemann surfaces}, }{Comm. Pure Appl. Math. 55 (2002), 728--771.}

\bibitem{cherrier}{P. Cherrier, }{ {Problemes de Neumann non lineaires sur les varietés riemanniennes}, }{J. Funct. Anal. 57 (1984), 154--206.}

\bibitem{CruzRuiz}{S. Cruz-Bl\'{a}zquez, D. Ruiz, }{ {Prescribing Gaussian and geodesic curvatures on disks}, }{Adv. Nonlinear Stud. 18 (2018), no. 3, 453--468.}



%\bibitem{DLM}{F. Da Lio, L. Martinazzi,}{  {The nonlocal Liouville-type equation in $\R$ and conformal immersions of the disk with boundary singularities},} {Calc. Var. Partial Differential Equations  56:152, (2017). }

\bibitem{DLMR}{F. Da Lio, L. Martinazzi, T. Rivi\`ere, }{ {Blow-Up Analysis of a Nonlocal Liouville-Type Equation},}{ Analysis and PDE 8 (2015), No. 7, 1757--1805.} 

%\bibitem{dmls}{F. De Marchis, R. L\'opez-Soriano, } {Existence and non existence results for the singular {N}irenberg problem, }{Calc. Var. Partial Differential Equations, 55(2):Art. 36, 35, 2016.}
%
%
%\bibitem{dmlsr}{F. De Marchis, R. L\'opez-Soriano, D. Ruiz, } {Compactness, existence and multiplicity for the singular mean field problem with sign-changing potentials, }{J. Math. Pures Appl. 115 (2018), 237--267.}

\bibitem{deimling}{K. Deimling, }{Nonlinear Functional Analysis, }{Springer-Verlag 1985.}

\bibitem{espinar}{J. M. Espinar, J. A. G\'{a}lvez and P. Mira, }{Hypersurfaces in $\mathbb{H}^n+1$ and conformally invariant equations: the generalized Christoffel and Nirenberg problems, }{J. Eur. Math. Soc. 11 (2009), 903-939. }

%\bibitem{DingJostLi}{W. Ding, J. Jost, J. Li, G. Wang, } {The differential equation $\Delta u = 8\pi - 8\pi e^u$ on a compact Riemann surface, }{Asian J. Math. 1 (1997), no. 2, 230--248.}


%\bibitem{DjadliMalchiodi}{Z. Djadli, A. Malchiodi, } {Existence of conformal metrics with constant $Q$-curvature, }{Ann. of Math.  168  (2008),  no. 3, 813--858.}

\bibitem{galvez}{J.A. G{\'a}lvez, P. Mira, }{ {The Liouville equation in a half-plane},}{ J. Differential Equations 246 (2009), no. 11, 4173--4187.}

\bibitem{manuela}{M. Gehrig, }{Prescribed curvature on the boundary of the diks, }{PhD thesis, \href{https://doi.org/10.3929/ethz-b-000445412}{https://doi.org/10.3929/ethz-b-000445412}.} 



\bibitem{guo-hu}{K. Guo and S. Hu, }{Conformal deformation of metrics on subdomains of surfaces, }{J. Geom. Anal. 5 (1995), no. 3, 395-410.}

\bibitem{GuoLiu}{Y.X. Guo, J. Liu, }{ {Blow-up analysis for solutions of the Laplacian equation with exponential Neumann boundary condition in dimension two}, }{Commun. Contemp. Math. 8 (2006), 737--761.}

\bibitem{hamza}{H. Hamza, }{ {Sur les transformations conformes des vari\`et\`es riemanniennes \`a bord}, }{J. Funct. Anal. 92 (1990), no. 2, 403--447.}

	\bibitem{Han}{Z. C. Han, }{Prescribing Gaussian curvature on $\mathbb{S}^2$, }{Duke Math. J. 61 (1990), no. 3, 679-703.}
	
	\bibitem{han-li}{Z. C. Han, Y. Li, }{A note on the Kazdan-Warner type condition, }{Ann. Inst. H. Poincar\'{e} Anal. Non Lin\'{e}aire 13 (1996), no. 3, 283-292.}
	
	\bibitem{jlmr}{A. Jevnikar, R. L\'{o}pez-Soriano, M. Medina and D. Ruiz, }{Blow-up analysis of conformal metrics of the disk with prescribed Gaussian and geodesic curvatures, }{to appear in Analysis and PDE's,  \href{https://arxiv.org/abs/2004.14680}{arXiv:2004.14680}.}

	
\bibitem{preJi}{M. Ji, }{On the Nirenberg problem, }{In: Gu, Hu, Li (eds.) Differential Geometry and Related Topics, pp. 107–134. Proceedings of the Conference in Honour of Prof. Su Buchin on the Centenary of His Birth, World Scientific Publishing Company, Singapore 2001.}

\bibitem{Ji}{M. Ji, }{On positive scalar curvature on $\mathbb{S}^2$, }{Calc. Var. 19, 165-182 (2004).}

%\bibitem{HangWang}{F. Hang, X. Wang }{ {A new approach to some nonlinear geometric equations in dimension two}, }{Calc. Var. Part. Diff. 26 (2006), no. 1, 119--135.}



%\bibitem{Jim}{A. Jim\'enez,}{  {The Liouville equation in an annulus},}{ Nonlinear Analysis 75 (2012), 2090--2097.} 

%\bibitem{JWang}{J. Jost, G. Wang, }{   {Analytic aspects of the Toda system. I. A Moser-Trudinger inequality}, }{ Comm. Pure Appl. Math. 54 (2001), no. 11, 1289--1319.}

%\bibitem{jwz}{J. Jost, G. Wang, C. Zhou, }{Metrics of constant curvature on a Riemann surface with
%two corners on the boundary, }{Ann. Inst. H. Poincar\'{e} Anal. Non Lin\'{e}aire, 26(2):437-456, 2009.}

\bibitem{KW}
{J.L. Kazdan, F.W. Warner, }{ {Curvature functions for compact 2-manifolds},}{ Ann.of Math. 99 (1974), 14--47.} 

%\bibitem{li}{Y. Y. Li, }{Harnack type inequality: The method of moving planes, }{Comm. Math. Phys. 200 (1999), 421–444.}

%\bibitem{li-liu}{P. Li, J. Liu, } {Nirenberg's problem on the 2-dimensional hemi-sphere, }{Int. J. Math. 4 (1993), 927-939.}

\bibitem{LiLiu}{Y. Li, P. Liu, }{ {A Moser-Trudinger inequality on the boundary of a compact Riemann surface}, }{Math. Z. 250 (2005), no. 2, 363--386.}

\bibitem{li-sha} {Y.Y. Li, I. Shafrir, }{ {Blow--up analysis for solutions of $-\Delta u = V e^u$ in dimension two}, } {Indiana Univ. Math. J. 43 (1994), no. 4, 1255--1270.}

\bibitem{li-zhu}{Y.Y. Li, M. Zhu, }{Uniqueness theorems through the method of moving spheres, }{Duke Math. J. 80 (1995), no. 2, 383-417.}

%\bibitem{LWZ}{C.S. Lin, J. Wei, L. Zhang, }{ {Classification of blowup limits for SU(3) singular Toda systems.}, }{Anal. PDE 8 (2015), no. 4, 807--837.}

\bibitem{LiHu}{P. Liu, W. Huang,}{  {On prescribing geodesic curvature on $D^2$},}{ Nonlinear Analysis 60 (2005) 465--473.}



\bibitem{lsmr}{ R. L\'opez-Soriano, A. Malchiodi, D. Ruiz, } {Conformal metrics with prescribed gaussian and geodesic curvatures, }{to appear in Ann. Sci. \'{E}cole Normale Sup\'{e}rieure de Paris,  \href{https://arxiv.org/abs/1806.11533}{arXiv:1806.11533}}.


\bibitem{lsr}{R. L\'opez-Soriano and D. Ruiz, } {Prescribing the Gaussian curvature on a subdomain of $\mathbb{S}^2$ 
with Neumann boundary conditions, }{J. Geom. Anal. 26 (2016), no. 1, 630-644.}

%\bibitem{MaWei}{L. Ma, J. Wei, } { Convergence for a Liouville equation, }{Comment. Math. Helv. 76 (2001) 506--514}.



\bibitem{Moser}{J. Moser, } {On a non-linear Problem in Differential Geometry and Dynamical Systems, }{Academic Press, N.Y. (ed M. Peixoto) 1973.}



%\bibitem{Mos}{J. Moser, }{A sharp form of an inequality by N. Trudinger, }{Indiana Univ. Math. J., 20:1077..1092, 1970/71.}

%\bibitem{NolTar}{M. Nolasco, G. Tarantello, } {On a sharp Sobolev-type inequality on two-dimensional compact manifolds, }{Arch. Ration. Mech. Anal. 145 (1998), no. 2, 161--195.}

\bibitem{osgood}{B. Osgood, R. Phillips, P. Sarnak, }{Extremals of determinants of Laplacians, }{J. Functional Analysis 80 (1988), 148-211.}

\bibitem{Rothe}{E. H. Rothe, }{Introduction to various aspects of Degree Theory in Banach Spaces, }{Mathematical Surveys and Monographs 23, AMS, 1986.}
	
	\bibitem{struwe}{M. Struwe, }{A flow approach to Nirenberg's problem, }{
Duke Math. J. 128 (2005), no. 1, 19-64. }

%\bibitem{TarAnal}{G. Tarantello, } {Analytical, geometrical and topological aspects of a class of mean field equations on surfaces, }{Discrete Contin. Dyn. Syst. 28 (2010), no. 3, 931--973.}
%

%\bibitem{Tarhb}{G. Tarantello, }\emph{Analytical aspects of Liouville-type equations with singular sources, }{Stationary partial differential equations. Vol I, 491--592, Handb. Differ. Equ., North-Holland, Amsterdam, (2004)}


%\bibitem{Tru}{N. Trudinger }{On imbeddings into Orlicz spaces and some applications, }{J. Math. Mech., 17:473–483, 1967..}


%\bibitem{Tarbook}{G. Tarantello, } {Self-Dual Gauge Field Vortices: An Analytical Approach, }{PNLDE 72, Birkh\"{a}user Boston, Inc., Boston, MA, 2007.}

%\bibitem{yang}{Y. Yang, } {Solitons in Field Theory and Nonlinear Analysis, }{Springer-Verlag, 2001.}

\bibitem{Wang}{G. Wang, }{Nirenberg's problem on Domains in the 2-Sphere, }{J. Geom. Anal. 11 (2001), 717-726.}

	\bibitem{xu-yang}{X. Xu, P.C. Yang, }{Remarks on prescribing Gauss curvature, }{TAMS 336 (1993), 831-840.}



\bibitem{Zhang}{L. Zhang, }{ {Classification of conformal metrics on $\mathbb{R}^2_+$ with constant Gauss curvature and geodesic curvature on the boundary under various integral finiteness assumptions}, }{Calc. Var. Partial Differential Equations 16 (2003), no. 4, 405--430.}

\end{thebibliography}

\end{document}